\documentclass[12pt,leqno]{amsart}

\usepackage[colorlinks=true,linkcolor=blue,citecolor=blue,urlcolor=blue]{hyperref}
\usepackage{url}
\usepackage{array}
\usepackage{amssymb,amsmath,amsthm,latexsym}
\usepackage{cleveref}
\usepackage{booktabs}

\theoremstyle{theorem}
\newtheorem{theorem}{Theorem}[section]
\newtheorem{proposition}[theorem]{Proposition}
\newtheorem{problem}[theorem]{Problem}
\numberwithin{equation}{section}
\theoremstyle{remark}
\newtheorem{remark}[theorem]{Remark}
\newtheorem{example}[theorem]{Example}

\newcommand{\cp}{\mathbb{C}P}
\newcommand{\qq}{\mathbb{Q}}

\allowdisplaybreaks

\begin{document}

\title[Complete intersections with different Hodge numbers]{Examples of diffeomorphic complete intersections with different Hodge numbers}
{\thanks{This project was supported by the Natural Science Foundation of Tianjin City of China, No. 19JCYBJC30300.}
\author{Jianbo Wang, Zhiwang Yu, Yuyu Wang}
\address{(Jianbo Wang) School of Mathematics, Tianjin University, Tianjin 300350, China}
\email{wjianbo@tju.edu.cn}
\address{(Zhiwang Yu) School of Mathematics, Tianjin University, Tianjin 300350, China}
\email{yzhwang@tju.edu.cn}
\address{(Yuyu Wang) College of Mathematical Science, Tianjin Normal University, Tianjin 300387, China}
\email{wdoubleyu@aliyun.com}
\date{\today}
\maketitle

\begin{abstract}
In this paper, we give three pairs of complex 3-dim complete intersections and a pair of complex 5-dim complete intersections, and every pair of them is diffeomorphic but with different Hodge numbers. Moreover, the diffeomorphic complex 3-dim complete intersections  have different Chern mumbers $c_1^3, c_1c_2$.
\end{abstract}

\section{Introduction}

The question of the topological invariance of Hodge numbers was firstly raised by Hirzebruch in 1954. His problem list \cite{Hi1954} contains the following question about the Hodge and Chern numbers of algebraic varieties, listed there as {\bf Problem 31}.

\vskip 2mm
 {\bf Problem 31.} (Hirzebruch) Are the $h^{p,q}$ and the Chern characteristic numbers of an algebraic variety $V_n$ topological invariants of $V_n$? If not, determine all those linear combinations of the $h^{p,q}$ and the Chern characteristic numbers which are topological invariants.

\vskip 2mm
For K\"ahler manifolds, the Hodge numbers have the same properties and symmetries as on algebraic varieties, so it is still interesting to consider the {\bf Problem 31} on K\"ahler manifolds. The question related the {\bf Problem 31} has also been raised repeatedly in mathoverflow posting by 
\href{https://mathoverflow.net/questions/42744/diffeomorphic-k\%C3\%A4hler-manifolds-with-different-hodge-numbers}{S\'andor Kov\'acs}, 
\href{https://mathoverflow.net/questions/42709/question-about-hodge-number/42735#42735}{Yunhyung Cho}, \href{https://mathoverflow.net/questions/173066/hodge-numbers-of-diffeomorphic-complete-intersections}{David C}, asking:

\begin{problem}\label{mainProblem}
Whether the Hodge numbers of K\"ahler manifolds are diffeomorphism invariants. 
\end{problem}

More details about the progress and solutions of Hirzebruch's 1954 problem list refer to Kotschick's survey paper \cite{Kot-arxiv13}. The
{\bf Problem 31} for the Chern numbers was completely solved in \cite{Kot09, Kot12}.
Kotschick and Schreieder answered the {\bf Problem 31} and \Cref{mainProblem} for Hodge numbers and  mixed linear combinations of Hodge and Chern numbers in \cite[Theorem 3 \& 4]{KotSch13}.

Firstly, let's briefly introduce the Chern numbers and Hodge numbers.
Assume $V$ is a compact complex $n$-dim manifold. 
The Chern number is defined to be the integer 
\[
c_{i_1}\cdots c_{i_r}:= \langle c_{i_1}(TV)\cdots c_{i_r}(TV),[V]\rangle,
\]
where $i_1+\cdots+i_r=n$, $TV$ is the tangent bundle of $V$, and $[V]$ is the fundamental homology class determined by the preferred orientation. 
The top Chern number $c_n$ is the Euler characteristic. Thus, it is a diffeomorphism invariant. For complex surfaces, the two Chern numbers $c_2$ and $c_1^2$ are diffeomorphism invariants (\cite{Kot07}). In \cite{BoHi59}, Borel and Hirzebruch showed that the homogeneous space $SU(4)/(SU(2) \times U(1) \times U(1))$ has two invariant structures as a Hodge manifold with different values for the Chern number $c_1^5$,
 which implies that $c_1^5$ is not a diffeomorphism invariant. In \cite{KotTer09},  Kotschick and Terzi\'c generalized the example in \cite{BoHi59} to $SU(n+2)/(SU(n) \times U(1) \times U(1))$, which determines different Chern numbers $c_1^{2n+1}$. Kotschick showed that both $c_1c_2$ and $c^3_1$ are not diffeomorphism invariants of simply connected projective algebraic 3-folds (\cite[Proposition 1]{Kot07}). 
 
Let $\Omega_V^p$ denote the  sheaf of holomorphic $p$-forms. Then the $(p,q)$-th Hodge number is
\[
h^{p,q}(V):=\dim H^{q}(V, \Omega_{V}^{p}). 
\]
Note that $h^{p,q}(V)=0$ if  $p$ or $q$ is greater than $n$. 
Suppose that $V$ is a compact K\"ahler manifold, then its Hodge numbers satisfy the K\"ahler symmetries: 
\[
h^{q,p}(V)=h^{p,q}(V)=h^{n-p,n-q}(V).
\]
The $r$-th Betti number $b_r(V)$ satisfies that
\[
b_r(V)=\sum_{p+q=r}h^{p,q}(V). 
\]
Thus, 
\[
b_0(V) =h^{0,0}(V),~b_1(V) =2h^{0,1}(V). 
\]
It implies that $h^{0,0}(V)$ and $h^{0,1}(V)$ are topological invariants for any compact K\"ahler manifold  modulo the K\"ahler symmetries. 
It is also well known that Hodge numbers are topological invariants of compact K\"ahler curves and surfaces. 

In \cite{Xiao}, Xiao found two simply connected complex surfaces $S$ and $S^{\prime}$ with different Hodge numbers, which are homeomorphic but not diffeomorphic. Moreover, in \cite{Campana},  Campana pointed out that $S\times S$ and $S^{\prime}\times S^{\prime}$ are orientedly diffeomorphic, and of course still have different Hodge numbers. 

However, just like Kotschick stated in \cite[p. 1301-1302]{Kot12}, there are few examples which show that certain Chern numbers and Hodge numbers are not diffeomorphism invariants. 
In \Cref{examples},  we present four such examples. These examples are related to the smooth complete intersections, which are compact K\"ahler manifolds. 
More concretely, 
\begin{enumerate}
\item \Cref{3-fold} contains three pairs of diffeomorphic complex $3$-dim complete intersections, which have different Chern numbers $c_1^3, c_1c_2$ and different Hodge numbers $h^{0,3}, h^{1,2}$. 
\item \Cref{5-fold} contains a pair of diffeomorphic complex $5$-dim complete intersections, which has different Hodge numbers $h^{2,3}$ and $h^{1,4}$.
\end{enumerate}
The Hodge numbers of complete intersections can be calculated by using the Sage software, and the code is offered by Pieter Belmans  in the link \url{https://github.com/pbelmans/hodge-diamond-cutter}.

{\bf Acknowledgements.} The authors would like to thank Pieter Belmans for sharing the Sage code to calculate the Hodge numbers of complete intersections.

\section{Chern numbers and Hodge numbers of complete intersections}

A complete intersection $X_{n}(\underline{d}) \subset \cp^{n+r}$ is the transversal intersection of $r$ complex hypersurfaces of degrees $d_1,\dots,d_r$, where $\underline{d}=(d_{1}, \dots, d_{r})$ is called the multidegree. The product of the degrees is called the total degree, which is denoted as $d:=d_1\cdots d_r$. The Chern numbers of $X_n(\underline{d})$ can be described as the polynomials of the dimension $n$ and the degrees $d_1,\dots,d_r$ (\cite{Wang}).
In particular, for a complete intersection $X_3(d_1,\dots,d_r)$, the Chern numbers $c_1c_2, c^3_1$ are as follows:
\begin{align*}\label{c13c1c2}
c_1^3 & =d\cdot\Big(4+r-\sum_{i=1}^rd_i\Big)^3,\\
c_1c_2 & =\frac{d}{2}\cdot\Big(4+r-\sum_{i=1}^rd_i\Big)\Big(\big(4+r-\sum_{i=1}^rd_i\big)^2-\big(4+r-\sum_{i=1}^rd_i^2\big)\Big).
\end{align*} 

Let $V$ be a compact complex $n$-dim manifold, the $p$-th holomorphic Euler number of $V$ is defined as:
\begin{equation*}
\chi^{p}(V)=\sum_{q=0}^{n}(-1)^{q} h^{p, q}(V).
\end{equation*}
The $\chi_y$-genus of $V$ is (\cite[\S 15.5]{Hi1978})
\[
\chi_y(V)=\sum_{p=0}^n\chi^p(V)y^p=\sum_{p=0}^n\sum_{q=0}^n (-1)^qh^{p,q}(V)y^p.
\]
The Hodge numbers $\{h^{p,q}(V)\}$ can be assembled into generating functions called the Hodge-Poincar\'e polynomials
\begin{equation*}
HP_{x, y}(V)=\sum_{p,q=0}^\infty h^{p,q}(V) x^{p} y^{q}.
\end{equation*}
By \cite[\S 22.1]{Hi1978}, the $\chi_y$-characteristic of trivial line bundle over $X_{n}(d_{1}, \dots, d_{r})$ is as follows:
\begin{align*}
& \sum_{n=0}^\infty \chi_{y}\left(X_{n}(d_{1}, \dots, d_{r})\right) z^{n+r}\\
=~& \frac{1}{(1+z y)(1-z)} \prod_{j=1}^{r} \frac{(1+z y)^{d_i}-(1-z)^{d_i}}{(1+z y)^{d_i}+y(1-z)^{d_i}}. 
\end{align*}

Every complete intersection is a smooth complex projective variety, and hence is a compact K\"ahler manifold. The Hodge numbers of $X_n(\underline{d})$ satisfy the K\"ahler constraints 
\[
h^{q,p}(X_n(\underline{d}))=h^{p,q}(X_n(\underline{d}))=h^{n-p,n-q}(X_n(\underline{d})).
\]
By \cite[Theorem 22.1.2]{Hi1978}, the Hodge number $h^{p,q}(X_n(\underline{d}))$ is trivially computable if $p+q \neq n $, which is given by 1 if $p=q\leqslant n$ and zero otherwise. The middle case, $p+q=n$ is more interesting. 
\begin{align*}
\chi^{p}\left(X_n(\underline{d})\right) & =(-1)^{n-p} h^{p, n-p}\left(X_n(\underline{d})\right)+(-1)^{p} \text { for } 2 p \neq n, \\
\chi^{m}\left(X_n(\underline{d})\right) & =(-1)^{m} h^{m, m}\left(X_n(\underline{d})\right) \text { for } 2m=n.
\end{align*}
The Theorem 2.3 in \cite[Exposé XI]{SGA7} shows that the degree $n=p+q$ part of $HP_{x, y}(X_n(\underline{d}))$ coincides with the degree $n$ part of the series
\begin{equation}\label{xyseries}
\frac{1}{(1+x)(1+y)}\left[\prod_{i=1}^r \frac{(1+x)^{d_{i}}-(1+y)^{d_{i}}}{(1+y)^{d_{i}} x-(1+x)^{d_{i}} y}-1\right]+\frac{1}{1-x y}.
\end{equation}
That is, $h^{p, q}(X_n(\underline{d}))$ is given by the coefficient of $x^{p} y^{q}$ in the power series \eqref{xyseries}.

\section{Examples}\label{examples}

The following diffeomorphic complete intersections can be found in \cite{Wang,WangDu}, and their Hodge numbers are calculated by the Sage software.

\begin{example} \label{3-fold}
The following three pairs of diffeomorphic complex 3-dim complete intersections have different Chern number $c_1^3$, $c_1c_2$ and different Hodge numbers $h^{0,3}$,  $h^{1,2}$ (\Cref{C3CI}). 
\end{example}	
\begin{table}[h]
\caption{$c_1^3, c_1c_2, h^{p,q}$ of diffeomorphic $3$-dim complete intersections}\label{C3CI} 
 \begin{tabular}{ccc}
  \toprule
  \null & $X_3(70, 16, 16, 14, 7, 6)$ & $X_3(56, 49, 8, 6, 5, 4, 4)$   \\
  \midrule
 $c_1^3$ & -17756452976640 & -18666867394560 \\
 $c_1c_2$ & -12441178337280 & -12956267089920 \\
 $h^{0,3}=h^{3,0}$ &  518382430721 &  539844462081 \\
 $h^{1,2}=h^{2,1}$ & 3365330286081 & 3343868254721 \\
  \bottomrule \\
  \toprule
  \null & $X_3(88, 28, 19, 14, 6, 6)$ & $X_3(76,56,11,7,6,6,2)$   \\
  \midrule
 $c_1^3$ & -81237337784064 & -84508254851328 \\
 $c_1c_2$ & -56913522817536 & -58764807230976 \\
 $h^{0,3}=h^{3,0}$ & 2371396784065 & 2448533634625 \\
 $h^{1,2}=h^{2,1}$ & 15351477911617 & 15274341061057 \\
  \bottomrule\\
  \toprule
  \null & $X_3(84, 29, 25, 25, 18, 7)$ & $X_3(60,58,49,9,5,5,5)$   \\
  \midrule
 $c_1^3$ & -1081901824920000 & -1118781720000000 \\
 $c_1c_2$ & -703318138110000 &  -723582436500000 \\
 $h^{0,3}=h^{3,0}$ & 29304922421251 & 30149268187501 \\
 $h^{1,2}=h^{2,1}$ & 162963432843751 & 162119087077501 \\
  \bottomrule
 \end{tabular}
\end{table}

\begin{example}\label{5-fold}
The complete intersections $X_{5}(66,56,45,39,16,15,8,3)$ and $X_{5}(64,60,42,39,20,11,9,3)$ are diffeomorphic, but they have different Hodge numbers $h^{1,4}=h^{4,1}$ and $h^{2,3}=h^{3,2}$ (\Cref{C5CI}). 
\end{example}

\begin{table}[h]
\caption{$h^{p,q}$ of diffeomorphic $5$-dim complete intersections} 
\label{C5CI}
\begin{tabular}{ccc}
  \toprule
  \null & $X_5(66,56,45,39,16,15,8,3)$ & $X_5(64,60,42,39,20,11,9,3)$   \\
  \midrule
 $h^{0,5}=h^{5,0}$ & 12092504413135386241 & 12092504413135386241 \\
 $h^{1,4}=h^{4,1}$ & 142963762408652465281 & 142965808209158014081 \\
 $h^{2,3}=h^{3,2}$ & 433117264328460391681 & 433115218527954842881 \\
  \bottomrule
 \end{tabular}
\end{table}

\begin{proposition}\label{prop-cn3}
For complex $3$-dim complete intersections and any $a,b\in\mathbb{Q}$ with $(a,b)\neq(0,0)$, the linear combination of Chern numbers represented as $ac_1^3+bc_1c_2 $ is not a diffeomorphism invariant. 
\end{proposition}
\begin{proof}
Suppose that there exists certain $a$, $b$ satisfying that $ac_1^3+bc_1c_2$ is  a diffeomorphism invariant. Given the following notations:
\begin{align*}
& X_{11}:=X_3(70,16,16,14,7,6), && X_{12}:=X_3(56,49,8,6,5,4,4); \\
& X_{21}:=X_3(88,28,19,14,6,6), && X_{22}:=X_3(76,56,11,7,6,6,2);\\
& X_{31}:=X_3(84,29,25,25,18,7), && X_{32}:=X_3(60,58,49,9,5,5,5). 
\end{align*}
By \Cref{3-fold}, $X_{i1}$ is diffeomorphic to $X_{i2}$. Let $c_{i_1}\cdots c_{i_r}[X_{ij}]$ be the Chern number of $X_{ij}$, so we have 
\[
ac_1^3[X_{i1}]+bc_1c_2[X_{i1}]=ac_1^3[X_{i2}]+bc_1c_2[X_{i2}], ~i=1,2,3.
\]
That is 
\[
\frac{b}{a}=-\frac{c_1^3[X_{i2}]-c_1^3[X_{i1}]}{c_1c_2[X_{i2}]-c_1c_2[X_{i1}]},~i=1,2,3.
\]
However by \Cref{C3CI},  
\begin{align*}
& \frac{c_1^3[X_{12}]-c_1^3[X_{11}]}{c_1c_2[X_{12}]-c_1c_2[X_{11}]}=\frac{43201}{24442}, \\ & \frac{c_1^3[X_{22}]-c_1^3[X_{21}]}{c_1c_2[X_{22}]-c_1c_2[X_{21}]}=\frac{69313}{39230}, \\ 
& \frac{c_1^3[X_{32}]-c_1^3[X_{31}]}{c_1c_2[X_{32}]-c_1c_2[X_{31}]}=\frac{96124}{52817}.
\end{align*} 
It is a contradiction! So $ac_1^3+bc_1c_2~(a, b\in\qq)$ is not a diffeomorphism invariant for complex $3$-dim complete intersections. 
\end{proof}
In \cite[Theorem 3]{Kot07}, Kotschick gave the following conclusion on the linear combinations of Chern numbers for simply connected $3$-folds. 

\textbf{Theorem 3.} (Kotschick) {\it The only linear combinations of the Chern numbers $c_1^3$ , $c_1 c_2$ and $c_3$ that are invariant under orientation-preserving diffeomorphisms of simply connected projective
algebraic three-folds are the multiples of the Euler characteristic $c_3$}.
\begin{remark}
As an application of \Cref{prop-cn3}, we can give a new proof of \cite[Theorem 3]{Kot07}. Note that the Chern number $c_3$ is the Euler characterisitic, and hence is a diffeomorphism invariant. 
Complete intersections are simply connected projective algebraic varieties. Therefore, by \Cref{prop-cn3}, we have already arrived at the same conclusion as \cite[Theorem 3]{Kot07} of Kotschick's. 
\end{remark}


\begin{thebibliography}{00}
\bibitem{BoHi59} A. Borel, F. Hirzebruch, {\it Characteristic classes and homogeneous spaces}, II, Amer. J. Math. 81 (1959) 315--382.
\bibitem{Campana} F. Campana, {\it Une remarque sur les nombres de Hodge des vari\'et\'es projectives complexes}, unpublished.
\bibitem{SGA7} P. Deligne, N. Katz, {\it Groupes de monodromie en g\'om\'etri\'e alge ́brique}, Lecture notes Mathematics, Vol. 340. Springer-Verlag, Berlin, 1973.
\bibitem{Hi1954} F. Hirzebruch, {\it Some problems on diﬀerentiable and complex manifolds}, Ann. of Math. (2) 60 (1954) 213--236.
\bibitem{Hi1978} F. Hirzebruch, {\it Topological methods in algebraic geometry}, Reprint of the 1978 Edition. Springer Science \& Business Media, 2012.
\bibitem{Kot07} D. Kotschick, {\it Chern numbers and diffeomorphism types of projective varieties}, J. Topol. 1 (2008) 518--526.
\bibitem{Kot09} D. Kotschick, {\it Characteristic numbers of algebraic varieties}, Proc. Natl. Acad. USA 106 (2009) 10114--10115.
\bibitem{Kot12} D. Kotschick, {\it Topologically invariant Chern numbers of projective varieties}, Adv. Math. 229 (2012) 1300--1312.
\bibitem{Kot-arxiv13} D. Kotschick, {\it Updates on Hirzebruch's 1954 Problem List}, arXiv:1305.4623v1, \url{https://arxiv.org/abs/1305.4623v1}
\bibitem{KotSch13} D. Kotschick, S. Schreieder, {\it The Hodge ring of K\"ahler manifolds}, Com- pos. Math. 149 (2013) 637--657.
\bibitem{KotTer09} D. Kotschick, S. Terzi\'c, {\it Chern numbers and the geometry of partial flag manifolds}, Comment. Math. Helv. 84 (2009) 587--616. 
\bibitem{Wang} J. Wang, {\it Remarks on 5-dimensional complete intersections}, Electron Res. Announc. Math Sci. 21 (2014) 28--40.
\bibitem{WangDu} J. Wang, J. Du, {\it Geometrical Realization of Low-Dimensional Complete Intersections}, Chin. Ann. Math. 37B(4) (2016) 523--532.
\bibitem{Xiao} G. Xiao, {\it An example of Hyperelliptic Surfaces with positive index}, North-eastern Math. J. 3 (1986) 255--257.
\end{thebibliography}
\end{document}